\documentclass[12pt,reqno]{amsart}
\usepackage{latexsym,amsmath,amsfonts,amssymb,amsthm}
\def\pmod #1{\ ({\rm{mod}}\ #1)}
\def\Z{\mathbb Z}

\def\Q{\mathbb Q}
\def\C{\mathbb C}
\def\1{{\mathbf 1}}

\def\pmod #1{\ ({\rm{mod}}\ #1)}

\theoremstyle{plain}
\newtheorem{Thm}{Theorem}
\newtheorem{Lem}{Lemma}

\newtheorem{Conj}{Conjecture}
\theoremstyle{definition}
\newtheorem*{Ack}{Acknowledgment}
\theoremstyle{remark}

\begin{document}

\title{Some $3$-adic congruences for binomial sums}
\author{Yong Zhang}
\email{yongzhang1982@163.com}
\address{Department of Basic Course, Nanjing Institute of Technology,
Nanjing 211167, People's Republic of China}
\author{Hao Pan}
\email{haopan79@yahoo.com.cn}
\address{Department of Mathematics, Nanjing University,
Nanjing 210093, People's Republic of China}
\begin{abstract}
We prove some $3$-adic congruences for binomial sums, which were conjectured by Sun.
\end{abstract}
\maketitle

\section{Introduction}
\setcounter{equation}{0}
\setcounter{Thm}{0}
\setcounter{Lem}{0}
\setcounter{Cor}{0}
\setcounter{Conj}{0}

For a non-zero integer $n$ and a prime $p$, let $\nu_p(n)$ denote the $p$-adic order of $n$, i.e., $\nu_p(n)$ is the largest integer such that $p^{\nu_p(n)}\mid n$.
In \cite{SSZ}, Strauss, Shallit and Zagier proved that for any positive integer $n$,
\begin{equation}\label{ssz}
\nu_3\bigg(\sum_{k=0}^{n-1}\binom{2k}{k}\bigg)=2\nu_3(n)+\nu_3\bigg(\binom{2n}{n}\bigg).
\end{equation}
Recently Sun \cite{SSZ} showed that
\begin{equation}\label{sun}
\nu_p\bigg(\sum_{k=0}^{n-1}\frac{1}{m^{k}}\binom{2k}{k}\bigg)\geq\nu_p(n)\quad\text{and}\quad \nu_p\bigg(\sum_{k=0}^{n-1}\frac{(-1)^k}{m^{k}}\binom{2k}{k}\binom{n-1}{k}\bigg)\geq\nu_p(n),
\end{equation}
where $m$ is an integer and $p$ is an odd prime dividing $m-4$. Furthermore, he also proposed several conjectures on the $3$-adic orders of the above two binomial sums.
\begin{Conj}
\label{sunc}
Let $m$ be integer with $m\equiv 1\pmod{3}$.\medskip

\noindent{\rm (i)} For every positive integer $n$, we have
\begin{equation}
\label{scc1}
\nu_{3}\bigg(\frac{1}{n}\sum_{k=0}^{n-1}\frac{\binom{2k}{k}}{m^{k}}\bigg)\geq
min\{\nu_{3}(n), \nu_{3}(m-1)-1\}.
\end{equation}
For any integer $a\geq \nu_{3}(m-1)$, we have
\begin{equation}
\label{scc2}
\frac{1}{3^{a}}\sum_{k=0}^{3^{a}-1}\frac{1}{m^{k}}\binom{2k}{k}\equiv\frac{m-1}{3}\pmod{3^
{\nu_{3}(m-1)}}.
\end{equation}
\noindent{\rm (ii)} For every positive integer $n$, we have
\begin{equation}
\label{scc3}
\nu_{3}\bigg(\frac{1}{n}\sum_{k=0}^{n-1}\frac{(-1)^k} {m^k}\binom{n-1}{k}
\binom{2k}{k}\bigg)\geq
min\{\nu_{3}(n), \nu_{3}(m-1)\}-1.
\end{equation}
For any integer $a>\nu_{3}(m-1)$, we have
\begin{equation}
\label{scc4}
\frac{1}{3^{a}}\sum_{k=0}^{3^{a}-1}\frac{(-1)^k}{m^k}\binom{3^{a}-1}{k}\binom{2k}{k}
\equiv-\frac{m-1}{3}\pmod{3^
{\nu_{3}(m-1)}}.
\end{equation}
\noindent{\rm (iii)} For any integer $a\geq 2$, we have
\begin{equation}
\label{scc5}
\frac{1}{3^{a}}\sum_{k=0}^{3^{a}-1}(-1)^k\binom{3^{a}-1}{k}
\binom{2k}{k}\equiv-3^{a-1}\pmod{3^{a}}.
\end{equation}
\end{Conj}
In this paper, we shall confirm Conjecture.
\begin{Thm}
All assertions of Conjecture \ref{sunc} are true.
\end{Thm}
The proofs of (\ref{scc1})-(\ref{scc5}) will be given in the next sections.

\section{Proofs of (\ref{scc1}) and (\ref{scc2})}
\setcounter{equation}{0}
\setcounter{Thm}{0}
\setcounter{Lem}{0}
\setcounter{Cor}{0}
\setcounter{Conj}{0}

For $A,B\in\Z$, the Lucas sequence $\{u_n(A,B)\}$ are given by
$$
u_0(A,B)=0,\quad u_1(A,B)=1,\quad u_{n+1}(A,B)=Au_n(A,B)-Bu_{n-1}(A,B)\text{ for }n\geq 1.
$$
In particular, it is not difficult to check that $\{u_n(-1,1)\}_{n\geq 0}\in\{0,1,-1\}$ and $u_n(-1,1)\equiv n\pmod{3}$.
\begin{Lem} Suppose that $m\equiv1\pmod{3}$. Then we have
\begin{equation}\label{l1c1}
\frac{u_{n}(m-2,1)}{n}\equiv
\frac{u_{n}(-1,1)}{n}+\frac{m-1}{3}\binom{n-1}{2}
\pmod{3^{\nu_{3}(m-1)}}
\end{equation}
if $m\not=4$, and
\begin{equation}\label{l1c2}
\frac{u_{n}(2,1)}{n}\equiv
\frac{u_{n}(-1,1)}{n}
\pmod{3}.
\end{equation}
In particular, we always have
\begin{equation}\label{l1c3}
\frac{u_{n}(m-2,1)}{n}\equiv
\frac{u_{n}(-1,1)}{n}
\pmod{3^{\nu_{3}(m-1)-1}}.
\end{equation}
\end{Lem}
\begin{proof}Let $\Delta=m(m-4)$.
By the properties of Lucas sequences, we have
\begin{equation}\label{lucas}
u_{n}(m-2,1)=\frac{1}{2^{n-1}}\sum_{\substack{1\leq k\leq n
\\2\nmid
k}}\frac{n}{k}\binom{n-1}{k-1}(m-2)^{n-k}\Delta^{(k-1)/2}.
\end{equation}
If $\Delta=0$, i.e., $m=4$, then
$$
\frac{u_{n}(2,1)}{n}=\bigg(\frac{4-2}{2}\bigg)^{n-1}=1\equiv\frac{u_{n}(-1,1)}{n}\pmod{3}.
$$
Suppose that $\Delta\not=0$. Then
\begin{align*}
\label{cd}
&\frac{u_{n}(m-2,1)}{n}-\bigg(\frac{m-2}{2}\bigg)^{n-1}=\sum_{\substack{1<
k\leq n
\\2\nmid
k}}\binom{n-1}{k-1}\bigg(\frac{m-2}{2}\bigg)^{n-k}\frac{\Delta^{(k-1)/2}}{k2^{k-1}}\\
=&\sum_{\substack{1< k\leq n
\\2\nmid
k}}\binom{n-1}{k-1}\frac{(m-2)^{n-k}}{2^{n-k}}\cdot\frac{((m-1)(m-3)-3)^{(k-1)/2}}{k2^{k-1}}\\
=&\sum_{\substack{1< k\leq n
\\2\nmid
k}}\binom{n-1}{k-1}\frac{(m-2)^{n-k}}{k2^{n-1}}\sum_{j=0}^{(k-1)/2}\binom{(k-1)/2}{j}(m-1)^j(m-3)^j{(-3)}^{(k-1)/2-j}
\\
\equiv&\sum_{\substack{1< k\leq n
\\2\nmid
k}}\binom{n-1}{k-1}\bigg(-\frac{1}{2}\bigg)^{n-k}\frac{{(-3)}^{(k-1)/2}}{k2^{k-1}}
+\frac{m-1}{3}\binom{n-1}{2}\pmod{3^{\nu_{3}(m-1)}}.
\end{align*}
By (\ref{lucas}), it is derived that
\begin{align*}
&\sum_{\substack{1< k\leq n
\\2\nmid
k}}\binom{n-1}{k-1}\bigg(-\frac{1}{2}\bigg)^{n-k}\frac{{(-3)}^{(k-1)/2}}{k2^{k-1}}\\
=&\frac{u_{n}(1-2,1)}{n}-\bigg(-\frac{1}{2}\bigg)^{n-1}
\equiv\frac{u_{n}(m-2,1)}{n}-\bigg(\frac{m-2}{2}\bigg)^{n-1}\pmod{3^{\nu_{3}(m-1)}}.
\end{align*}
We are done.
\end{proof}
The following curious identity is due to Sun and Taurso \cite[(2.1)]{ST}:
\begin{equation}\label{st1}
m^{n-1}\sum_{k=0}^{n-1}\frac{1} {m^k}\binom{2k}{k}=\sum_{k=0}^{n-1}\binom{2n}{k}u_{n-k}(m-2,1).
\end{equation}
It is easy to check that
\begin{equation}\label{n2nk}
\frac1n\binom{2n}{k}=\frac1{n-k}\bigg(2
\binom{2n-1}{k}-\binom{2n}{k}\bigg).
\end{equation}
So (\ref{st1}) can be rewritten as
\begin{equation} \label{st2}
\frac{m^{n-1}}{n}\sum_{k=0}^{n-1}\frac{1}{m^k}\binom{2k}{k}
=\sum_{k=0}^{n-1} \bigg(2
\binom{2n-1}{k}-\binom{2n}{k}\bigg)\frac{u_{n-k}(m-2,1)}{n-k}.
\end{equation}
Thus using (\ref{l1c3}) and (\ref{n2nk}), we have
\begin{align*}
\frac{m^{n-1}}{n}\sum_{k=0}^{n-1}\frac{1} {m^k}\binom{2k}{k}
\equiv&\sum_{k=0}^{n-1} \bigg(2
\binom{2n-1}{k}-\binom{2n}{k}\bigg)\frac{u_{n-k}(-1,1)}{n-k}\\
=&\frac{1}{n}\sum_{k=0}^{n-1}\binom{2n}ku_{n-k}(-1,1)=\frac{1}{n}
\sum_{k=0}^{n-1}\binom{2k}{k}\pmod{3^{\nu_{3}(m-1)-1}}.
\end{align*}
Thus (\ref{scc1}) easily follows from (\ref{ssz}).

Suppose that $a\geq\nu_{3}(m-1)$. When $m=4$, we have $u_{n}(2,1)=n$. By (\ref{st1}),
\begin{align*}
\frac{1}{3^a}\sum_{k=0}^{3^{a}-1} \frac{1}
{4^k}\binom{2k}{k}=&\frac{1}{3^{a}\cdot4^{3^a-1}}
\sum_{k=0}^{3^{a}-1}\binom{2\cdot3^a}{k}(3^a-k)\\
=&\frac{1}{4^{3^a-1}}\sum_{k=0}^{3^{a}-1}\binom{2\cdot3^a}{k}-\frac{2}{4^{3^a-1}}\sum_{k=1}^{3^{a}-1}\binom{2\cdot3^a-1}{k-1}\\
\equiv&{\sum_{k=0}^{3^{a}-1}\binom{2\cdot3^a}{k}+\sum_{k=1}^{3^{a}-1}\binom{2\cdot3^a-1}{k-1}}\equiv{1}\pmod{3}.
\end{align*}
Suppose that $m\not=4$.
Note that
\begin{align*}&\sum_{k=0}^{3^{a}-1}\bigg(2\binom{2\cdot3^a-1}{k}-\binom{2\cdot3^a}{k}\bigg)\cdot\frac{u_{3^a-k}(-1,1)}{3^a-k}\\\
=&\frac{1}{3^a}\sum_{k=0}^{3^{a}-1}\binom{2\cdot3^a}{k}u_{3^a-k}(-1,1)=\frac{1}{3^{a}}
\sum_{k=0}^{3^a-1}\binom{2k}{k}\equiv0\pmod{3^a}.
\end{align*}
Hence applying (\ref{l1c1}) and (\ref{st2}), we get
\begin{align*}
\frac{1}{3^a}\sum_{k=0}^{3^{a}-1} \frac{1}
{m^k}\binom{2k}{k}
=&\frac{1}{m^{3^{a}-1}}\sum_{k=0}^{3^{a}-1}\bigg(2\binom{2\cdot3^a-1}{k}-\binom{2\cdot3^a}{k}\bigg)\cdot\frac{u_{3^a-k}(m-2,1)}{3^a-k}\\
\equiv&\sum_{k=0}^{3^{a}-1}\bigg(2\binom{2\cdot3^a-1}{k}-\binom{2\cdot3^a}{k}\bigg)\cdot\frac{m-1}{3}\binom{3^a-k-1}{2}\\
\equiv&\frac{m-1}{3}\bigg(\sum_{k=0}^{3^{a}-1}2\binom{2\cdot3^a-1}{k}\binom{2\cdot3^a-k-1}{2}-\binom{3^a-1}{2}\bigg)\\
=&\frac{m-1}{3}\bigg(\sum_{k=0}^{3^a-1}(2\cdot3^a-1)(2\cdot3^a-2)\binom{2\cdot3^{a}-3}{k}-\binom{3^a-1}{2}\bigg)\\
\equiv&\frac{m-1}{3}\pmod{3^{\nu_{3}(m-1)}},
\end{align*}
where in the last step we use the fact
$$
\sum_{k=0}^{3^a-1}\binom{2\cdot3^{a}-3}{k}=2^{2\cdot3^{a}-4}+\frac12\binom{2\cdot3^{a}-3}{3^{a}-2}+\binom{2\cdot3^{a}-3}{3^{a}-1}\equiv 1\pmod{3}.
$$
This proves (\ref{scc2}).

\section{Proofs of (\ref{scc3}) and (\ref{scc5})}
\setcounter{equation}{0}
\setcounter{Thm}{0}
\setcounter{Lem}{0}
\setcounter{Cor}{0}
\setcounter{Conj}{0}

In this section, we shall prove (\ref{scc3}) and (\ref{scc4}), under the assumption of the following congruence:
\begin{equation}\label{nk2kk}
\sum_{k=0}^{n-1}(-1)^k\binom{n-1}k\binom{2k}{k}\equiv0\pmod{3^{2\nu_3(n)-1}}.
\end{equation}
And the proof of (\ref{nk2kk}) will be given in the final section. We also need a special case of an identity of Sun \cite[(2.6)]{ZWS}:
\begin{equation}
\frac{1}{n}\sum_{k=0}^{n-1}\frac{(-1)^k}
{m^k}\binom{n-1}{k}
\binom{2k}{k}=\sum_{k=1}^{n}\frac{(-1)^{k-1}}{k}\binom{n-1}{k-1}\sum_{l=0}^{k-1}\frac{1}{m^l}\binom{2l}{l}.
\end{equation}
Thus from (\ref{l1c3}), (\ref{st2})and (\ref{nk2kk}), it follows that
\begin{align*}
&\frac{1}{n}\sum_{k=0}^{n-1}\frac{(-1)^k}
{m^k}\binom{n-1}{k}
\binom{2k}{k}\\
=&\sum_{k=1}^{n}\frac{(-1)^{k-1}}{m^{k-1}}\binom{n-1}{k-1}\sum_{l=0}^{k-1}\bigg(2
\binom{2k-1}{l}-\binom{2k}{l}\bigg)\frac{u_{k-l}(m-2,1)}{k-l}\\
\equiv&\sum_{k=1}^{n}\frac{(-1)^{k-1}}{m^{k-1}}\binom{n-1}{k-1}\sum_{l=0}^{k-1}\bigg(2
\binom{2k-1}{l}-\binom{2k}{l}\bigg)\frac{u_{k-l}(-1,1)}{k-l}\\
\equiv&\frac{1}{n}\sum_{k=0}^{n-1}\frac{(-1)^{k-1}}{m^{k-1}}\binom{n-1}{k}
\binom{2k}{k}\equiv 0\pmod{3^{\nu_{3}(m-1)-1}}.
\end{align*}
So (\ref{scc3}) is concluded.

First, suppose that $m\not=4$. For an integer $a\geq\nu_{3}(m-1)+1$, define
\begin{align*}
f(a)=&{\sum_{k=1}^{3^{a}}\frac{(-1)^{k-1}}{m^{k-1}}\binom{3^{a}-1}{k-1}\sum_{l=0}^{k-1}\bigg(2
\binom{2k-1}{l}-\binom{2k}{l}\bigg)}\binom{k-l-1}{2}.\\
\end{align*}
By (\ref{l1c1}),
\begin{align*}
&\sum_{k=1}^{3^a}\frac{(-1)^{k-1}}{m^{k-1}}\binom{3^a-1}{k-1}\sum_{l=0}^{k-1}\bigg(2
\binom{2k-1}{l}-\binom{2k}{l}\bigg)\frac{u_{k-l}(m-2,1)}{k-l}\\
\equiv&\sum_{k=1}^{3^a}\frac{(-1)^{k-1}}{m^{k-1}}\binom{3^a-1}{k-1}\sum_{l=0}^{k-1}\bigg(2
\binom{2k-1}{l}-\binom{2k}{l}\bigg)\frac{u_{k-l}(-1,1)}{k-l}\\
&+\sum_{k=1}^{3^a}\frac{(-1)^{k-1}}{m^{k-1}}\binom{3^a-1}{k-1}\sum_{l=0}^{k-1}\bigg(2
\binom{2k-1}{l}-\binom{2k}{l}\bigg)\cdot\frac{m-1}{3}\binom{k-l-1}{2}\\
=&\frac{1}{3^a}\sum_{k=0}^{3^a-1}\frac{(-1)^{k-1}}{m^{k-1}}\binom{3^a-1}{k}
\binom{2k}{k}+\frac{m-1}{3}f(a)\pmod{3^{\nu_{3}(m-1)}}.
\end{align*}
Thus in view of (\ref{nk2kk}), it suffices to show that
$$f(a)\equiv-1\pmod{3}.
$$
Noting that
\begin{align*}
\binom{k-l-1}{2}=\frac{(k-l-1)(k-l-2)}{2}\equiv\begin{cases}1\pmod{3},&\text{if }3\mid k-l,\\
0\pmod{3},&\text{otherwise},\end{cases}
\end{align*}
we have
\begin{align*}
&\sum_{l=0}^{k-1}\bigg(2
\binom{2k-1}{l}-\binom{2k}{l}\bigg)\binom{k-l-1}{2}\\
\equiv&\sum_{\substack{0\leq j\leq k-l\\ 3\mid k-l}}\bigg(2
\binom{2k-1}{l}-\binom{2k}{l}\bigg)\pmod{3}.
\end{align*}
By the proof of \cite[Theorem 1.1]{ZWS},
$$\sum_{\substack{0\leq j\leq k-l\\ 3\mid k-l}}\bigg(2
\binom{2k-1}{l}-\binom{2k}{l}\bigg)\equiv{\binom{2k/{3^{\nu_{3}(k)}}-1}{{k}/{3^{\nu_{3}(k)}}-1}}\pmod{3}.$$
Apparently for $0\leq k\leq 3^a-1$,
$$
\binom{3^a-1}{k}=\prod_{j=1}^k\bigg(\frac{3^a}{j}-1\bigg)\equiv (-1)^k\pmod{3}.
$$
Hence noting that $a\geq 2$ and applying (\ref{ssz}), we can get
\begin{align*}
f(a)\equiv&{\sum_{k=1}^{3^{a}}(-1)^{k-1}\binom{3^{a}-1}{k-1}\binom{2k/{3^{\nu_{3}(k)}}-1}{{k}/{3^{\nu_{3}(k)}}-1}}\equiv\sum_{j=1}^{a}\sum_{\substack{1\leq i\leq 3^{a-j}\\ 3\nmid i}}\binom{2i-1}{i-1}\\
=&1+\frac12\sum_{j=1}^{a-1}\bigg(\sum_{i=0}^{3^{a-j}-1}\binom{2i}{i}-\sum_{i=0}^{3^{a-j-1}-1}\binom{6i}{3i}\bigg)\\
\equiv&1+\frac12\sum_{j=1}^{a-1}\bigg(\sum_{i=0}^{3^{a-j}-1}\binom{2i}{i}-\sum_{i=0}^{3^{a-j-1}-1}\binom{2i}{i}\bigg)
\equiv1-\frac12\equiv{-1}\pmod{3}.
\end{align*}

Finally, if $m=4$, then we have
\begin{align*}
\frac{1}{3^a}\sum_{k=0}^{3^{a}-1}\frac{(-1)^k} {4^k}\binom{3^{a}-1}{k}
\binom{2k}{k}
=&\sum_{k=1}^{3^{a}}\binom{3^{a}-1}{k-1}\frac{(-1)^{k-1}}{k}\sum_{l=0}^{k-1}\frac{1}{4^l}\binom{2l}{l}\\
\equiv&\sum_{k=1}^{3^{a}}\frac1k\sum_{l=0}^{k-1}\binom{2l}{l}\pmod{3}.\end{align*}
Clearly,
$$
\sum_{l=0}^{k-1}\frac{1}{4^l}\binom{2l}{l}=\sum_{l=0}^{k-1}(-1)^l\binom{-1/2}{l}=(-1)^{k-1}\binom{-1/2-1}{k-1}=\frac{k}{2^{2k-1}}\binom{2k}{k}.
$$
So using (\ref{ssz}) again, we obtain that
\begin{align*}
\frac{1}{3^a}\sum_{k=0}^{3^{a}-1}\frac{(-1)^k} {4^k}\binom{3^{a}-1}{k}
\binom{2k}{k}\equiv2\sum_{k=1}^{3^{a}}\binom{2k}{k}\equiv2\bigg(\binom{2\cdot3^a}{3^a}-1\bigg)\equiv-1\pmod{3}.
\end{align*}

\section{Proofs of (\ref{scc5}) and (\ref{nk2kk})}
\setcounter{equation}{0}
\setcounter{Thm}{0}
\setcounter{Lem}{0}
\setcounter{Cor}{0}
\setcounter{Conj}{0}
The key of the proofs of (\ref{scc5}) and (\ref{nk2kk}) is the following identity.
\begin{Lem}
\begin{equation}\label{nk2kki}
\sum_{k=0}^{n}\binom{2k}{k}\binom{n}k(-x)^k=\frac{1}{4^{n}}\sum_{k=0}^{n}\binom{2j}{j}\binom{2(n-j)}{n-j}(1-4x)^k.
\end{equation}
\end{Lem}
\begin{proof}
\begin{align*}
\sum_{k=0}^{n}\binom{2k}{k}\binom{n}k(-x)^k=&\sum_{k=0}^{n}\binom{-1/2}{k}\binom{n}k(4x-1+1)^k\\
=&\sum_{k=0}^{n}\binom{-1/2}{k}\binom{n}k\sum_{j=0}^k\binom{k}j(4x-1)^j\\
=&\sum_{j=0}^n\binom{-1/2}{j}(4x-1)^j\sum_{k=j}^{n}\binom{-1/2-j}{k-j}\binom{n}{n-k}\\
=&\sum_{j=0}^n\binom{-1/2}{j}\binom{n-1/2-j}{n-j}(4x-1)^j\\
=&\frac1{4^{n}}\sum_{j=0}^n\binom{2j}{j}\binom{2(n-j)}{n-j}(1-4x)^j.
\end{align*}
\end{proof}
Substituting $x=1$ in (\ref{nk2kki}), we get
\begin{equation}\label{nk2kki1}
\sum_{k=0}^{n-1}(-1)^k\binom{2k}{k}\binom{n-1}k=\frac1{4^{n-1}}\sum_{k=0}^{n-1}(-3)^k\binom{2k}{k}\binom{2(n-1-k)}{n-1-k}.
\end{equation}
Let $a=\nu_3(n)$. Obviously,
$$
\sum_{k=0}^{n-1}(-3)^k\binom{2j}{j}\binom{2(n-1-k)}{n-1-k}\equiv\sum_{k=0}^{2a-2}(-3)^k\binom{2k}{k}\binom{2(n-1-k)}{n-1-k}\pmod{3^{2a-1}}.
$$
When $a=1$, writing $n=3b$ with $3\nmid b$, we have
$$
\sum_{k=0}^{0}(-3)^k\binom{2k}{k}\binom{2(n-1-k)}{n-1-k}=\binom{6b-2}{3b-1}
$$
is divisible by $3$. Below we only consider $a\geq 2$.
Noting that $$
\binom{2(n-1)}{n-1}\frac{\binom{n-1}{k}^2}{\binom{2(n-1)}{2k}}=\binom{2k}{k}\binom{2(n-1-k)}{n-1-k},
$$
we have
$$
\sum_{k=0}^{n-1}(-1)^k\binom{2k}{k}\binom{n-1}k\equiv\frac1{4^{n-1}}\binom{2(n-1)}{n-1}\sum_{k=0}^{2a-2}(-3)^k\frac{\binom{n-1}{k}^2}{\binom{2(n-1)}{2k}}\pmod{3^{2a-1}}.
$$
Since
$$
\binom{2(n-1)}{n-1}=\frac{n}{2(2n-1)}\binom{2n}{n}\equiv 0\pmod{3^a},
$$
it suffices to prove that
$$
\sum_{k=0}^{2a-2}(-3)^k\frac{\binom{n-1}{k}^2}{\binom{2(n-1)}{2k}}\equiv 0\pmod{3^{a-1}}.
$$
Clearly,
$$
\frac{\binom{n-1}{k}^2}{\binom{2(n-1)}{2k}}=\frac{1}{2k+1}\cdot\frac{\prod_{j=1}^k(1-n/j)^2}{\prod_{j=2}^{2k+1}(1-2n/j)}.
$$
For $1\leq k\leq 2a-2$ and $2\leq j\leq 2k+1$, it is easy to check that $$\nu_3(j)\leq a-1$$ and
$$
\nu_3((2k+1)j)\leq k+1.
$$
Hence for $2\leq j\leq k$
$$
\frac{3^k(1-n/j)}{2k+1}=\frac{3^k}{2k+1}-\frac{3^kn}{(2k+1)j}\equiv\frac{3^k}{2k+1}\pmod{3^{a-1}}.
$$
Similarly, for $1\leq j\leq 2k+1$, we also have
$$
\frac{3^k}{(2k+1)(1-2n/j)}\equiv\frac{3^k}{2k+1}\pmod{3^{a-1}},
$$
since
$$
\frac{1}{1-2n/j}=1+\frac{2n}{j}+\bigg(\frac{2n}{j}\bigg)^2+\cdots
$$
over the rational $3$-adic field $\Q_3$. Thus we get
$$
\sum_{k=0}^{2a-2}(-3)^k\frac{\binom{n-1}{k}^2}{\binom{2(n-1)}{2k}}\equiv\sum_{k=0}^{2a-2}\frac{(-3)^k}{2k+1}\pmod{3^{a-1}}.
$$
Note that for $k\geq 2a-1$, we always have
$$
\frac{3^k}{2k+1}\equiv 0\pmod{3^{a-1}}.
$$
Thus (\ref{nk2kk}) immediately follows from the following lemma.
\begin{Lem}
$$
\sum_{k=0}^\infty\frac{(-3)^k}{2k+1}
$$
vanishes over $\Q_3$.
\end{Lem}
\begin{proof}
Let $\C_3$ denote the completion of the algebraic closure of $\Q_3$. For any $x\in\C_3$ with the $3$-adic norm $|x|_3<1$, define the $3$-adic logarithm function
$$
\log_3(1+x)=\sum_{n=1}^\infty\frac{(-1)^{n+1}x^n}{n}.
$$
Clearly,
\begin{align*}
\sum_{k=0}^\infty\frac{(-3)^k}{2k+1}=&\frac{1}{\sqrt{-3}}\sum_{k=0}^\infty\frac{(\sqrt{-3})^{2k+1}}{2k+1}=\frac{1}{2\sqrt{-3}}\bigg(\sum_{k=1}^\infty\frac{(\sqrt{-3})^{k}}{k}-\sum_{k=1}^\infty\frac{(-\sqrt{-3})^{k}}{k}\bigg)\\
=&\frac{1}{2\sqrt{-3}}(\log_3(1+\sqrt{-3})-\log_3(1-\sqrt{-3}))
=\frac{1}{2\sqrt{-3}}\log_3\bigg(\frac{\sqrt{-3}-1}{2}\bigg).
\end{align*}
Since $(\sqrt{-3}-1)/2$
is a third root of unity, the lemma is concluded.
\end{proof}
The proof of (\ref{scc5}) is very similar, only requiring a few additional discussions. Now we have
$$
\sum_{k=0}^{3^a-1}(-1)^k\binom{2k}{k}\binom{3^a-1}k\equiv\frac1{4^{3^a-1}}\binom{2(3^a-1)}{3^a-1}\sum_{k=0}^{2a-1}(-3)^k\frac{\binom{3^a-1}{k}^2}{\binom{2(3^a-1)}{2k}}\pmod{3^{2a}}.
$$
Since
$$
\frac{1}{3^a}\binom{2(3^a-1)}{3^a-1}=\frac{1}{2(2\cdot3^a-1)}\binom{2\cdot3^a}{3^a}\equiv-1\pmod{3},
$$
we only need to show that
$$
\sum_{k=0}^{2a-1}(-3)^k\frac{\binom{3^a-1}{k}^2}{\binom{2(3^a-1)}{2k}}\equiv3^{a-1}\pmod{3^{a}}.
$$
Since $a\geq 2$, for $1\leq k\leq 2a-1$ and $1\leq j\leq 2k+1$, we always have $\nu_3(j)\leq a-1$, and we also have
$\nu_3((2k+1)j)\leq k$ unless $k=1$ and $j=3$. Hence for $(k,j)\not=(1,3)$,
$$
\frac{3^k(1-n/j)}{2k+1}\equiv\frac{3^k}{2k+1}\pmod{3^{a}}\quad\text{and}\quad\frac{3^k}{(2k+1)(1-2n/j)}\equiv\frac{3^k}{2k+1}\pmod{3^{a}}.
$$
That is, for $k\geq 2$,
$$
\frac{\binom{3^a-1}{k}^2}{\binom{2(3^a-1)}{2k}}=\frac{1}{2k+1}\cdot\frac{\prod_{j=1}^k(1-n/j)^2}{\prod_{j=2}^{2k+1}(1-2n/j)}\equiv\frac{(-3)^k}{2k+1}\pmod{3^{a}}.
$$
It follows that
\begin{align*}
&\sum_{k=0}^{2a-1}(-3)^k\frac{\binom{3^a-1}{k}^2}{\binom{2(3^a-1)}{2k}}\equiv
\sum_{k=0}^{1}(-3)^k\frac{\binom{3^a-1}{k}^2}{\binom{2(3^a-1)}{2k}}+\sum_{k=2}^{2a-1}\frac{(-3)^k}{2k+1}\\
\equiv&1+(-3)\cdot\frac{\binom{3^a-1}{1}^2}{\binom{2(3^a-1)}{2}}=1-\frac{6(3^a-1)^2}{(2\cdot3^a-2)(2\cdot3^a-3)}\\
=&\frac{3^{a-1}}{1-2\cdot 3^{a-1}}\equiv3^{a-1}\pmod{3^a}.
\end{align*}
\begin{Ack}
We are grateful to Professor Zhi-Wei Sun for his helpful discussions on this paper.
\end{Ack}

\end{document}